\documentclass[a4paper,11pt,oneside,reqno]{amsart}
\usepackage{amssymb,amsfonts,amsmath,amsthm,amscd}
\usepackage{graphicx, color}

\numberwithin{equation}{section}  

\providecommand{\U}[1]{\protect\rule{.1in}{.1in}}
\newtheorem{theorem}{Theorem}[section]

\newtheorem{conjecture}[theorem]{Conjecture}
\newtheorem{corollary}[theorem]{Corollary}

\newtheorem{definition}[theorem]{Definition}

\newtheorem{lemma}[theorem]{Lemma}

\newtheorem{proposition}[theorem]{Proposition}

\numberwithin{equation}{section}
\makeindex

\begin{document}
\title[Quenched LDP and Parisi formula for GREM Perceptron]{A quenched large deviation principle and a Parisi formula for a perceptron
version of the GREM.}

\author[E. Bolthausen]{Erwin Bolthausen}            
 \address{E. Bolthausen\\ Institut f\"ur Mathematik \\
Universit\"at Z\"urich \\
Winterthurerstrasse 190 \\
CH-8057 Z\"urich}
\email{eb@math.uzh.ch}

\author[N. Kistler]{Nicola Kistler}
\address{N. Kistler\\Institut f\"ur Angewandte Mathematik\\Rheinische
   Friedrich-Wilhelms-Universit\"at Bonn \\Endenicher Allee 60\\ 53115
   Bonn,
Germany}
\email{nkistler@uni-bonn.de}


\thanks{E. Bolthausen is supported in part by the grant No $200020_125247/1$ of the
Swiss Science Foundation. N. Kistler is partially supported by the German Research Council in the SFB 611 and the Hausdorff Center for Mathematics.} 

\begin{abstract}
We introduce a perceptron version of the Generalized Random Energy Model, and
prove a quenched Sanov type large deviation principle for the empirical
distribution of the random energies. The dual of the rate function has a
representation through a variational formula which is closely related to the
Parisi variational formula for the SK-model.
\end{abstract}

\date{\today}

\maketitle

\hfill\emph{Dedicated to J\"{u}rgen G\"{a}rtner on the occasion of his 60th birthday.}

\section{Introduction \label{introduction}}

There has been important progress in the mathematical study of mean field spin
glasses over the last $10$ years. By results of Guerra \cite{guerra} and
Talagrand \cite{TalagrandParisi}, the free energy of the
Sherrington-Kirkpatrick model is known to be given by the formula predicted by
Parisi \cite{parisi}. Furthermore, the description of the \textit{high}
temperature is remarkably accurate, see \cite{talagrand} and references
therein. On the other hand, results for the Gibbs measure at \textit{low}
temperature are more scarce and are restricted to models with a simpler
structure, like Derrida's generalized random energy model, the GREM,
\cite{BovierKurkova} and \cite{DerridaGREM}, the nonhierarchical GREMs
\cite{bokis_two} and the $p$-spin model with large $p$ \cite{talagrand}. To
put on rigorous ground the full Parisi picture remains a major challenge, and
even more so in view of its alleged universality, at least for mean-field models.

We introduce here a model which hopefully sheds some new light on the issue.

In this paper we derive the free energy, which can be analyzed by large
deviation techniques. The limiting free energy turns out to be given by a
Gibbs variational formula which can be linked to a Parisi-type formula
by a duality principle, so that it becomes evident why an infimum appears in
the latter. This duality also gives an interesting interpretation of the
Parisi order parameter in terms of the sequence of inverse of temperatures
associated to the extremal measures from the Gibbs variational principle.

In a forthcoming paper, we will give a full description of the Gibbs measure
in the thermodynamic limit in terms of the Ruelle cascades.

\section{A Perceptron version of the GREM\label{Sect_perceptron}}

Let $\left\{  X_{\alpha,i}\right\}  _{\alpha\in\Sigma_{N},1\leq i\leq N},$ be
random variables which take values in a Polish space $S$ equipped with the
Borel $\sigma$-field $\mathcal{S},$ and defined on a probability space
$\left(  \Omega,\mathcal{F},\mathbb{P}\right)  .$ We write $\mathcal{M}%
_{1}^{+}\left(  S\right)  $ for the set of probability measures on $\left(
S,\mathcal{S}\right)  ,$ which itself is a Polish space. $\Sigma_{N}$ is
exponential in size, typically $\left\vert \Sigma_{N}\right\vert =2^{N}.$ It
is assumed that all $X_{\alpha,i}$ have the same distribution $\mu$, and that
for any fixed $\alpha\in\Sigma_{N},$ the collection $\left\{  X_{\alpha
,i}\right\}  _{1\leq i\leq N}$ is independent. It is however not assumed that
they are independent for different $\alpha.$ The perceptron Hamiltonian is
defined by%
\begin{equation}
-H_{N,\omega}\left(  \alpha\right)  \overset{\mathrm{def}}{=}\sum_{i=1}%
^{N}\phi\left(  X_{\alpha,i}\left(  \omega\right)  \right)
,\label{general_perceptron}%
\end{equation}
where $\phi:S\rightarrow\mathbb{R}$ is a measurable function. One may allow
that the index set for $i$ is rather $\left\{  1,\ldots,\left[  aN\right]
\right\}  $ with $a$ some positive real number, but for convenience, we always
stick to $a=1$ here. The case which is best investigated (see \cite{talagrand}%
) takes for $\alpha$ spin sequences: $\alpha=\left(  \sigma_{1},\ldots
,\sigma_{N}\right)  \in\left\{  -1,1\right\}  ^{N},$ $S=\mathbb{R},$ and the
$X_{\alpha,i}$ are centered Gaussians with%
\begin{equation}
\mathbb{E}\left(  X_{\alpha,i}X_{\alpha^{\prime},i^{\prime}}\right)
=\delta_{i,i^{\prime}}\frac{1}{N}\sum_{j=1}^{N}\sigma_{j}\sigma_{j}^{\prime
}.\label{SK_perceptron}%
\end{equation}
This is closely related to the SK-model, and is actually considerably more
difficult. The model has been investigated by Talagrand \cite{talagrand}, but
a full Parisi formula for the free energy is lacking.

The Hamiltonian (\ref{general_perceptron}) can be written in terms of the
empirical measure%
\begin{equation}
L_{N,\alpha}\overset{\mathrm{def}}{=}\frac{1}{N}\sum_{i=1}^{N}\delta
_{X_{\mathbf{\sigma},i}}\label{empirical_Def}%
\end{equation}
i.e.%
\[
-H_{N,\omega}\left(  \alpha\right)  =N\int\phi\left(  x\right)  L_{N,\alpha
}\left(  dx\right)  .
\]

The quenched free energy is the almost sure limit of%
\[
\frac{1}{N}\log\sum_{\alpha}\exp\left[  -H_{N,\omega}\left(  \alpha\right)
\right]  ,
\]
and it appears natural to ask if this free energy can be obtained by a
quenched Sanov type large deviation principle for $L_{N,\alpha}$ in the
following form:

\begin{definition}
We say that $\left\{  L_{N}\right\}  $ satisfies a \textbf{quenched large
deviation principle} (in short QLDP) with good rate function $J:\mathcal{M}%
_{1}^{+}\left(  S\right)  \rightarrow\left[  -\infty,\infty\right)  ,$
provided the level sets of $J$ are compact, and for any weakly continuous
bounded map $\Phi:\mathcal{M}_{1}^{+}\left(  S\right)  \rightarrow\mathbb{R},
$ one has%
\[
\lim_{N\rightarrow\infty}\frac{1}{N}\log\sum_{\alpha\in\Sigma_{N}}\exp\left[
N\Phi\left(  L_{N,\alpha}\right)  \right]  =\log2+\sup_{\nu\in\mathcal{M}%
_{1}^{+}\left(  S\right)  }\left[  \Phi\left(  \nu\right)  -J\left(
\mu\right)  \right]  ,,\ \mathbb{P}\mathrm{-a.s.}%
\]

\end{definition}

The annealed version of such a QLDP is just Sanov's theorem:%
\begin{align*}
\lim_{N\rightarrow\infty}\frac{1}{N}\log\sum_{\alpha}\mathbb{E}\exp\left[
N\Phi\left(  L_{N,\alpha}\right)  \right]   & =\log2+\lim_{N\rightarrow\infty
}\frac{1}{N}\log\mathbb{E}\exp\left[  N\Phi\left(  L_{N,\alpha}\right)
\right] \\
& =\log2+\sup_{\nu}\left(  \Phi\left(  \nu\right)  -H\left(  \nu%
\vert
\mu\right)  \right)
\end{align*}
where $H\left(  \nu%
\vert
\mu\right)  $ is the usual relative entropy of $\nu$ with respect to $\mu,$
the latter being the distribution of the $X_{\alpha,i}:$%
\[
H\left(  \nu%
\vert
\mu\right)  \overset{\mathrm{def}}{=}\left\{
\begin{array}
[c]{cc}%
\int\log\frac{d\nu}{d\mu}\ d\nu & \mathrm{if\ }\nu\ll\mu\\
\infty & \mathrm{otherwise}%
\end{array}
\right.  .
\]
There is no reason to believe that $H\left(  \nu%
\vert
\mu\right)  =J\left(  \nu\right)  .$

\begin{conjecture}
The empirical measures $\left\{  L_{N,\alpha}\right\}  $ with
(\ref{SK_perceptron}) satisfy a QLDP.
\end{conjecture}

We don't know how this conjecture could be proved, nor do we have a clear
picture what $J$ should be in this case. The only support we have for the
conjecture is that it is true in a perceptron version of the GREM, a model we
are now going to describe.

For $n\in{\mathbb{N}}$, $\alpha=(\alpha_{1},\dots,\alpha_{n})$ with
$1\leq\alpha_{k}\leq2^{\gamma_{i}N}$, $\sum_{k}\gamma_{k}=1$, and $1\leq i\leq
N$, let
\[
X_{\alpha,i}=\left(  X_{\alpha_{1},i}^{1},X_{\alpha_{1},\alpha_{2},i}%
^{2},\dots,X_{\alpha_{1},\alpha_{2},\dots,\alpha_{n},i}^{n}\right)
\]
where the $X^{j}$ are independent, taking values in some Polish Space
$(S,{\mathcal{S}})$ with distribution $\mu_{j}$. For notational convenience,
we assume that the $\gamma_{i}N$ are all integers. Put%
\[
\Gamma_{j}\overset{\mathrm{def}}{=}\sum_{k=1}^{j}\gamma_{j}.
\]
We assume that all the variables in the bracket are independent. The
$X_{\alpha,i}$ take values in $S^{n}.$ The distribution is%
\[
\mu\overset{\mathrm{def}}{=}\mu_{1}\otimes\cdots\otimes\mu_{n}%
\]
The empirical measure $L_{N,\alpha}$ is defined by (\ref{empirical_Def}) which
is a random element in ${\mathcal{M}}_{1}^{+}(S^{n})$. $n$ is fixed in all we
are doing.

Given a measure $\nu\in{\mathcal{M}}_{1}^{+}(S^{n})$, and $1\leq j\leq n,$ we
write $\nu^{(j)}$ for its marginal on the first $j$ coordinates. We define
subsets $\mathcal{R}_{j}$ of ${\mathcal{M}}_{1}^{+}(S^{n})$, $1\leq j\leq n$
by%
\[
\mathcal{R}_{j}\overset{\mathrm{def}}{=}\left\{  \nu\in{\mathcal{M}}_{1}%
^{+}(S^{n}):H\left(  \nu^{(j)}\mid\mu^{(j)}\right)  \leq\Gamma_{j}%
\log2\right\}  .
\]
We will also consider the sets%
\[
\mathcal{R}_{j}^{=}\overset{\mathrm{def}}{=}\left\{  \nu\in{\mathcal{M}}%
_{1}^{+}(S^{n}):H\left(  \nu^{(j)}\mid\mu^{(j)}\right)  =\Gamma_{j}%
\log2\right\}  .
\]

For $\nu\in{\mathcal{M}}_{1}^{+}(S^{n})$ let%
\[
J\left(  \nu\right)  =\left\{
\begin{array}
[c]{cc}%
H(\nu\mid\mu) & \mathrm{if\ }\nu\in\bigcap\nolimits_{j=1}^{n}\mathcal{R}_{j}\\
\infty & \mathrm{otherwise}%
\end{array}
\right.  .
\]
It is evident that $J$ is convex and has compact level sets.

Our first main result is:

\begin{theorem}
\label{Th_GREM_perceptron} $\left\{  L_{N,\alpha}\right\}  $ satisfies a QLDP
with rate function $J.$
\end{theorem}

For the rest of this section, we will focus on linear functionals, $\Phi
(\nu)=\int\phi(x)\nu({d}x)$, for a bounded continuous function $\phi
:S^{n}\rightarrow\mathbb{R}.$ For a probability measure $\nu$ on $S^{n}$, we
set%
\[
\operatorname{Gibbs}(\phi,\nu)\overset{\mathrm{def}}{=}\int\phi(x)\nu
(dx)-H(\nu\mid\mu),
\]
and define the Legendre transform of $J$ by%
\[
J^{\ast}\left(  \phi\right)  \overset{\mathrm{def}}{=}\sup_{\nu}\left[
\int\phi(x)\nu(dx)-J\left(  \nu\right)  \right]  =\sup\left\{
\operatorname{Gibbs}(\phi,\nu):\nu\in\bigcap\nolimits_{j=1}^{n}\mathcal{R}%
_{j}\right\}  .
\]
whenever the a.s.-limit exists. As a corollary of Theorem
\ref{Th_GREM_perceptron} we have

\begin{corollary}
\label{Cor_GREMperceptron_linear}Assume that $\phi:S\rightarrow{\mathbb{R}} $
is bounded and continuous.%
\[
\lim_{N\rightarrow\infty}\frac{1}{N}\log\sum_{\alpha}\exp\left[
\sum\nolimits_{i=1}^{N}\phi\left(  X_{\alpha,i}\right)  \right]  =J^{\ast
}\left(  \phi\right)  +\log2,\ \mathrm{a.s.}%
\]

\end{corollary}

We next discuss a dual representation of $J^{\ast}\left(  \phi\right)  $.
Essentially, this comes up by investigating which measures solve the
variational problem. Remark that without the restrictions $\nu\in
\bigcap\nolimits_{j=1}^{n}\mathcal{R}_{j},$ we would simply get%
\[
d\nu=\frac{\mathrm{e}^{\phi}d\mu}{\int\mathrm{e}^{\phi}d\mu}%
\]
as the maximizer.

Let $\Delta$ be the set of sequences $\mathbf{m}=\left(  m_{1},\ldots
,m_{n}\right)  $ with $0<m_{1}\leq m_{2}\leq\cdots\leq m_{n}\leq1.$ For
$\mathbf{m}\in\Delta,$ and $\phi:S^{n}\rightarrow\mathbb{R}$ bounded, we
define recursively functions $\phi_{j},~0\leq j\leq n,\ \phi_{j}%
:S^{j}\rightarrow\mathbb{R},$ by%
\begin{equation}
\phi_{n}\overset{\mathrm{def}}{=}\phi,\label{Def_phin}%
\end{equation}%
\begin{equation}
\phi_{j-1}\left(  x_{1},\ldots,x_{j-1}\right)  \overset{\mathrm{def}}{=}%
\frac{1}{m_{j}}\log\int\operatorname{exp}\left[  {m}_{j}\phi_{j}\left(
x_{1},\dots,x_{j-1},x_{j}\right)  \right]  \mu_{j}\left(  dx_{j}\right)
.\label{Def_phij}%
\end{equation}
$\phi_{0}$ is just a real number, which we denote by $\phi_{0}\left(
\mathbf{m}\right)  .$

Remark that if some of the $m_{i}$ agree, say $m_{k}=m_{k+1}=\cdots=m_{l},$
$k<l,$ then $\phi_{k-1}$ is obtained from $\phi_{l}$ by%
\[
\phi_{k-1}\left(  x_{1},\ldots,x_{k-1}\right)  =\frac{1}{m_{k}}\log
\int\operatorname{exp}\left[  {m}_{k}\phi_{l}\left(  x_{1},\dots,x_{k-1}%
,x_{k},\ldots,x_{l}\right)  \right]  \prod\limits_{j=k}^{l}\mu_{j}\left(
dx_{j}\right)  .
\]
In particular, if all the $m_{i}$ are $1,$ then%
\[
\phi_{0}=\log\int\exp\left[  \phi\right]  d\mu.
\]
This latter case corresponds to the \textquotedblleft replica
symmetric\textquotedblright\ situation. Put%
\begin{equation}
\operatorname{Parisi}\left(  \mathbf{m},\phi\right)  \overset{\mathrm{def}}%
{=}\sum\nolimits_{i=1}^{n}{\frac{\gamma_{i}\log2}{m_{i}}}+\phi_{0}\left(
\mathbf{m}\right)  -\log2\label{Parisi}%
\end{equation}

\begin{theorem}
\label{Th_Parisi_formula}Assume that $\phi:S\rightarrow{\mathbb{R}}$ is
bounded and continuous. Then%
\begin{equation}
J^{\ast}\left(  \phi\right)  =\inf_{\mathbf{m}\in\Delta}\operatorname{Parisi}%
\left(  \mathbf{m},\phi\right)  .\label{Variationformula3}%
\end{equation}

\end{theorem}

The expression for $J^{\ast}\left(  \phi\right)  $ in this theorem is very
similar to the Parisi formula for the SK-model. Essentially the only
difference is the first summand which in the SK-case is a quadratic
expression. In our case (in contrast to the still open situation in the
SK-model), we can prove that the infimum is uniquely attained, as we will
discuss below.

The derivation of the theorem from Corollary \ref{Cor_GREMperceptron_linear}
is done by identifying first the possible maximizers in the variational
formula for $J^{\ast}\left(  \phi\right)  $. They belong to a family of
distributions, parametrized by $\mathbf{m}.$ The maximizer inside this family
is then obtained by minimizing $\mathbf{m}$ according to
(\ref{Variationformula3}), and one then identifies the two expressions. The
procedure is quite standard in large deviation situations.

Two conventions: $C$ stands for a generic positive constant, not necessarily
the same at different occurences. If there are inequalities stated between
expressions containing $N,$ it is tacitely assumed that they are valid maybe
only for large enough $N.$

\section{Proofs}

\subsection{The Gibbs variational principle: Proof of Theorem
\ref{Th_GREM_perceptron}}

If $A\in{\mathcal{S}}$, we put $H(A\mid\mu)\overset{\mathrm{def}}%
{=}\operatorname{inf}_{\nu\in A}H(\nu\mid\mu)$. If $S$ is a Polish Space, and
${\mathcal{S}}$ its Borel $\sigma$-field, then it is well known that
$\nu\rightarrow H(\nu\mid\mu)$ is lower semicontinuous in the weak topology.
This follows from the representation
\begin{equation}
H(\nu\mid\mu)=\sup_{u\in{\mathcal{U}}}\left[  \int u\,d\nu-\log\int
\mathrm{e}^{u}d\mu\right]  ,\label{sup_representation}%
\end{equation}
where ${\mathcal{U}}$ is the set of bounded continuous functions
$S\rightarrow{\mathbb{R}}$.

For $(S,{\mathcal{S}}),(S^{\prime},{\mathcal{S}}^{\prime})$ two Polish Spaces,
and $\nu\in{\mathcal{M}}_{1}^{+}(S\times S^{\prime})$. If $\mu\in{\mathcal{M}%
}_{1}^{+}(S)$, $\mu^{\prime}\in{\mathcal{M}}_{1}^{+}(S^{\prime})$ we have,
\begin{equation}
H\left(  \nu\mid\mu\otimes\mu^{\prime}\right)  =H\left(  \nu^{(1)}\mid
\mu\right)  +H\left(  \nu\mid\nu^{(1)}\otimes\mu^{\prime}\right)
,\label{Entropy_Add}%
\end{equation}
where $\nu^{(1)}$ is the first marginal of $\nu$ on $S$.

\begin{lemma}
\label{lower_semicontinuity} $H(\nu\mid\nu^{(1)}\otimes\mu^{\prime})$ is a
lower semicontinuous function of $\nu$ in the weak topology.
\end{lemma}

\begin{proof}
Applying (\ref{sup_representation}) to
\[
H(\nu\mid\nu^{(1)}\otimes\mu^{\prime})=\sup_{u\in{\mathcal{U}}}\left[  \int
ud\nu-\log\int\mathrm{e}^{u}d\left(  \nu^{(1)}\otimes\mu^{\prime}\right)
\right]  ,
\]
where ${\mathcal{U}}$ denotes the set of bounded continuous functions $S\times
S^{\prime}\rightarrow{\mathbb{R}}$. For any fixed $u\in{\mathcal{U}}$, both
functions $\nu\rightarrow\int u\,d\nu$ and $\nu\rightarrow\log\int
\mathrm{e}^{u}d\left(  \nu^{(1)}\otimes\mu^{\prime}\right)  $ are continuous,
and from this the desired semicontinuity property follows.
\end{proof}

We will need the following \textquotedblleft relative\textquotedblright%
\ version of Sanov's theorem. Consider three independent sequences of i.i.d.
random variables $(X_{i}),(Y_{i}),(Z_{i})$, taking values in three Polish
spaces $S,S^{\prime},S^{\prime\prime},$ and with laws $\mu,\mu^{\prime}%
,\mu^{\prime\prime}$. We consider the empirical processes
\[
L_{N}\overset{\mathrm{def}}{=}{\frac{1}{N}}\sum_{i=1}^{N}\delta_{(X_{i}%
,Y_{i})},\ R_{N}\overset{\mathrm{def}}{=}{\frac{1}{N}}\sum_{i=1}^{N}%
\delta_{\left(  X_{i},Z_{i}\right)  }.
\]
The pair $(L_{N},R_{N})$ takes values in ${\mathcal{M}}_{1}^{+}(S\times
S^{\prime})\times{{\mathcal{M}}}_{1}^{+}(S\times S^{\prime\prime}).$

\begin{lemma}
\label{ldp_empirical_measure_couple} The sequence $(L_{N},R_{N})$ satisfies a
LDP with rate function
\[
J(\nu,\theta)=%
\begin{cases}
H\left(  \nu^{(1)}\mid\mu\right)  +H\left(  \nu\mid\nu^{(1)}\otimes\mu
^{\prime}\right)  +H\left(  \theta\mid\theta^{(1)}\otimes\mu^{\prime\prime
}\right)  , & \mathrm{if}\;\nu^{(1)}=\theta^{(1)}\\
\infty & \mathrm{otherwise}.
\end{cases}
\]

\end{lemma}

\begin{proof}
We apply the Sanov theorem to the empirical measure
\[
M_{N}={\frac{1}{N}}\sum_{i=1}^{N}\delta_{(X_{i},Y_{i},Z_{i})}\in{{\mathcal{M}%
}}_{1}^{+}(S\times S^{\prime}\times S^{\prime\prime}).
\]
We use the two natural projections $p:S\times S^{\prime}\times S^{\prime
\prime}\rightarrow S\times S^{\prime}$ and $q:S\times S^{\prime}\times
S^{\prime\prime}\rightarrow S\times S^{\prime\prime}$. Then $(L_{N}%
,R_{N})=M_{N}(p,q)^{-1}$, and by continuous projection, we get that
$(L_{N},R_{N})$ satisfies a good LDP with rate function
\[
J^{\prime}(\nu,\theta)=\operatorname{inf}\left\{  H(\rho\mid\mu\otimes
\mu^{\prime}\otimes\mu^{\prime\prime}):\rho p^{-1}=\nu,\rho q^{-1}%
=\theta\right\}  .
\]
It only remains to identify this rate function with the function $J$ given above.

Clearly $J^{\prime}(\nu,\theta)=\infty$ if $\nu^{(1)}\neq\theta^{(1)}$.
Therefore, assume $\nu^{(1)}=\theta^{(1)}$. If we define $\hat{\rho}\left(
\nu,\theta\right)  \in\mathcal{M}_{1}^{+}\left(  S\times S^{\prime}\times
S^{\prime\prime}\right)  $ to have marginal $\nu^{(1)}=\theta^{(1)}$ on $S$,
and the conditional distribution on $S^{\prime}\times S^{\prime\prime}$ given
the first projection is the product of the conditional distributions of $\nu$
and $\theta$, then applying twice (\ref{Entropy_Add}), we get%
\[
H(\hat{\rho}\mid\mu\otimes\mu^{\prime}\otimes\mu^{\prime\prime})=H\left(
\nu^{(1)}\mid\mu\right)  +H\left(  \nu\mid\nu^{(1)}\otimes\mu^{\prime}\right)
+H\left(  \theta\mid\theta^{(1)}\otimes\mu^{\prime\prime}\right)  ,
\]
and therefore $J\geq J^{\prime}$.

To prove the other inquality, consider any $\rho$ satisfying $\rho p^{-1}%
=\nu,\rho q^{-1}=\theta$. We want to show that $J(\nu,\theta)\leq H\left(
\rho\mid\mu\otimes\mu^{\prime}\otimes\mu^{\prime\prime}\right)  $. For that,
we can assume that the right hand side is finite. Then%
\[
H\left(  \rho\mid\mu\otimes\mu^{\prime}\otimes\mu^{\prime\prime}\right)
=H\left(  \rho\mid\hat{\rho}\left(  \nu,\theta\right)  \right)  +\int
d\rho\log\frac{d\hat{\rho}\left(  \nu,\theta\right)  }{d\left(  \mu\otimes
\mu^{\prime}\otimes\mu^{\prime\prime}\right)  }.
\]
The first summand is $\geq0,$ and the second equals%
\[
\int d\hat{\rho}\left(  \nu,\theta\right)  \log\frac{d\hat{\rho}\left(
\nu,\theta\right)  }{d\left(  \mu\otimes\mu^{\prime}\otimes\mu^{\prime\prime
}\right)  }=J(\nu,\theta).
\]
So, we have proved that%
\[
J(\nu,\theta)\leq H\left(  \rho\mid\mu\otimes\mu^{\prime}\otimes\mu
^{\prime\prime}\right)  ,
\]
for any $\rho$ satisfying $\rho p^{-1}=\nu,\rho q^{-1}=\theta.$
\end{proof}

We now step back to the setting of Theorem \ref{Th_GREM_perceptron}: For
$j=1,\dots,n,$ we have sequences $\left\{  X_{\alpha_{1},\dots,\alpha_{j}%
,i}^{j}\right\}  $ of independent random variables with distribution $\mu_{j}%
$. We emphasize that henceforth $\mu=\mu_{1}\otimes\cdots\otimes\mu_{n}$ and
$\mu^{(j)}$ will denote the marginal on the first $k$ components. Moreover,
for $\alpha=(\alpha_{1},\dots,\alpha_{n})$, we write $\alpha^{(j)}=(\alpha
_{1},\dots,\alpha_{j})$ and set
\[
L_{N,\alpha^{(j)}}^{(j)}={\frac{1}{N}}\sum_{i=1}^{N}\delta_{\left(
X_{\alpha_{1},i}^{1},X_{\alpha_{1},\alpha_{2},i}^{2},\dots,X_{\alpha_{1}%
,\dots,\alpha_{j},i}^{j}\right)  },
\]
for $j\leq n$, which is the marginal of $L_{N,\alpha}$ on $S^{j}$. With the
notation
\begin{align*}
& X_{\alpha,i}^{(j)}\overset{\mathrm{def}}{=}\left(  X_{\alpha_{1},i}%
^{1},\dots,X_{\alpha_{1},\dots,\alpha_{j},i}^{j}\right)  ,\\
& \hat{X}_{\alpha,i}^{(j)}\overset{\mathrm{def}}{=}\left(  X_{\alpha_{1}%
,\dots,\alpha_{j+1},i}^{j+1},\dots,X_{\alpha_{1},\dots,\alpha_{n},i}%
^{n}\right)  ,
\end{align*}
we can write
\begin{equation}
L_{N,\alpha}\overset{\mathrm{def}}{=}{\frac{1}{N}}\sum_{i=1}^{N}%
\delta_{\left(  X_{\alpha,i}^{(j)},\hat{X}_{\alpha,i}^{(j)}\right)
}.\label{splitting_empirical}%
\end{equation}
For $A\subset{{\mathcal{M}}}_{1}^{+}(S^{n})$ we put $M_{N}(A)\overset
{\mathrm{def}}{=}\#\left\{  \alpha:L_{N,\alpha}\in A\right\}  $.

\begin{lemma}
\label{very_many} Assume $\nu\in{\mathcal{M}}_{1}^{+}(S^{n})$ satisfies
$H(\nu\mid\mu)<\infty$, and let $V$ be an open neighborhood of $\nu$, and
$\varepsilon>0$. Then there exists an open neighborhood $U$ of $\nu$,
$U\subset V$, and $\delta>0$ such that
\[
{\mathbb{P}}\Big [M_{N}(U)\geq\operatorname{exp}\left[  N\left(  \log
2-H(\nu\mid\mu)+\varepsilon\right)  \right]  \Big ]\leq\mathrm{e}^{-\delta N}.
\]

\end{lemma}

\begin{proof}
If $B_{r}(\nu)$ denotes the open $r$-ball around $\nu$ in one of the standard
metrics, e.g. the Prohorov metric, then by the semicontinuity property of the
relative entropy, on has%
\[
H(B_{r}(\nu)\mid\mu)\uparrow H(\nu\mid\mu)
\]
as $r\downarrow0.$ We can choose a sequence $r_{k}>0,r_{k}\downarrow0$ with
$H(B_{r_{k}}(\nu)\mid\mu)=H(\operatorname{cl}\left(  B_{r_{k}}(\nu)\right)
\mid\mu)\uparrow H(\nu\mid\mu)$. Given $\varepsilon>0,$ and $V,$ we can find
$k$ such that%
\[
H(B_{r_{k}}(\nu)\mid\mu)=H(\operatorname{cl}\left(  B_{r_{k}}(\nu)\right)
\mid\mu)\geq H(\nu\mid\mu)-\varepsilon/4
\]
and $B_{r_{k}}(\nu)\subset V.$ By Sanov's theorem we therefore get
\[
{\mathbb{P}}\Big [L_{N,\alpha}\in B_{r_{k}}(\nu)\Big ]\leq\operatorname{exp}%
\left[  N(-H(\nu\mid\mu)+\varepsilon/2)\right]  ,
\]
and therefore%
\[
{{\mathbb{E}}}\Big [M_{N}\left(  B_{r_{k}}(\nu)\right)  \Big ]\leq
\operatorname{exp}\left[  N(\log2-H(\nu\mid\mu)+\varepsilon/2)\right]  .
\]
By the Markov inequality, the claim follows by taking $\delta=\varepsilon/3.$
\end{proof}

\begin{lemma}
\label{still_existent} Assume $\nu\in{\mathcal{M}}_{1}^{+}(S^{n})$ satisfies
$H\left(  \nu^{(j)}\mid\mu^{(j)}\right)  >\Gamma_{j}\log2$ for some $j\leq n$,
and let $V$ be an open neighborhood of $\nu$. Then there is an open
neighborhood $U$ of $\nu$, $U\subset V$ and $\delta>0$ such that
\[
{\mathbb{P}}\big [M_{N}(U)\neq0\big ]\leq\mathrm{e}^{-\delta N}%
\]
for large enough $N$.
\end{lemma}

\begin{proof}
As in the previous lemma, we choose a neighborhood $U^{\prime}$ of $\nu^{(j)}$
in $S^{j}$ such that $H(\operatorname{cl}\left(  U^{\prime}\right)  \mid
\mu^{(j)})=H(U^{\prime}\mid\mu^{(j)})>\Gamma_{j}\log2+\eta,$ for some
$\eta>0.$ Then we put
\[
U\overset{\mathrm{def}}{=}\left\{  \nu\in{\mathcal{M}}_{1}^{+}(S^{n}):\nu\in
V,\nu^{(j)}\in U^{\prime}\right\}  .
\]
If $L_{N,\alpha}\in U$ then $L_{N,\alpha}^{(j)}\in U^{\prime}$,
\begin{align*}
{\mathbb{P}}\left[  \exists\alpha:L_{N,\alpha}\in U\right]   & \leq
{\mathbb{P}}\left[  \exists\alpha:L_{N,\alpha}^{(j)}\in U^{\prime}\right] \\
& \leq2^{\Gamma_{j}N}{\mathbb{P}}\left[  L_{N,\alpha}^{(j)}\in U^{\prime
}\right] \\
& \leq2^{\Gamma_{j}N}\operatorname{exp}\left[  -NH\left(  \operatorname{cl}%
\left(  U^{\prime}\right)  \mid\mu^{(j)}\right)  +N\eta/2\right] \\
& \leq2^{\Gamma_{j}N}\operatorname{exp}\left[  -N\Gamma_{j}\log2-N\eta
/2\right]  =\mathrm{e}^{-N\eta/2}.
\end{align*}
This proves the claim.
\end{proof}

\begin{lemma}
\label{second_mom} Assume that $\nu\in{\mathcal{M}}_{1}^{+}(S^{n})$ satisfies
$H\left(  \nu^{(j)}\mid\mu^{(j)}\right)  <\Gamma_{j}\log2$ for all $j$, and
let $V$ be an open neighborhood of $\nu$, and $\varepsilon>0$. Then there
exists an open neighborhood $U$ of $\nu$, $U\subset V$, and a $\delta>0$ such
that
\[
{\mathbb{P}}\Big [M_{N}(U)\leq\operatorname{exp}\left[  N\left(  \log
2-H(\nu\mid\mu)-\varepsilon\right)  \right]  \Big ]\leq\mathrm{e}^{-\delta N}.
\]

\end{lemma}

\begin{proof}
We claim that we can find $U$ as required, and some $\delta>0,$ such that
\begin{equation}
{\operatorname{var}}\left[  M_{N}(U)\right]  \leq\mathrm{e}^{-2N\delta
}\left\{  {{\mathbb{E}}}\left[  M_{N}(U)\right]  \right\}  ^{2}%
\label{less_evident_sanov}%
\end{equation}
From this estimate, we easily get the claim: From Sanov's theorem, we have for
any $\chi>0$%
\begin{equation}
\mathbb{E}M_{N}(U)=2^{N}\mathbb{P}\left(  L_{N,\alpha}\in U\right)  \geq
\exp\left[  N\left(  \log2-H(\nu\mid\mu)-\chi\right)  \right]
.\label{Est_Exp_below}%
\end{equation}
Using this, we get by taking $\chi=\varepsilon/2$%
\begin{align*}
& {\mathbb{P}}\left(  M_{N}(U)\leq\mathrm{e}^{N\left(  \log2-H(\nu\mid
\mu)-\varepsilon\right)  }\right) \\
& ={\mathbb{P}}\left(  M_{N}(U)-\mathbb{E}M_{N}(U)\leq\mathrm{e}%
^{-N\varepsilon/2}\mathrm{e}^{N\left(  \log2-H(\nu\mid\mu)-\varepsilon
/2\right)  }-\mathbb{E}M_{N}(U)\right) \\
& \leq{\mathbb{P}}\left(  M_{N}(U)-\mathbb{E}M_{N}(U)\leq\left(
\mathrm{e}^{-N\varepsilon/2}-1\right)  \mathbb{E}M_{N}(U)\right) \\
& \leq{\mathbb{P}}\left(  M_{N}(U)-\mathbb{E}M_{N}(U)\leq-\frac{1}%
{2}\mathbb{E}M_{N}(U)\right) \\
& \leq{\mathbb{P}}\left(  \left\vert M_{N}(U)-\mathbb{E}M_{N}(U)\right\vert
\geq\frac{1}{2}\mathbb{E}M_{N}(U)\right) \\
& \leq4\frac{{\operatorname{var}}\left[  M_{N}(U)\right]  }{\left\{
\mathbb{E}M_{N}(U)\right\}  ^{2}}\leq4\mathrm{e}^{-2N\delta}\leq
\mathrm{e}^{-\delta N}.
\end{align*}

So it remains to prove (\ref{less_evident_sanov}). We first claim that for any
$j$%
\begin{align}
& \lim_{r\rightarrow0}\operatorname{inf}_{\rho,\theta\in{\operatorname{cl}%
}B_{r}(\nu):\rho^{(j)}=\theta^{(j)}}\left\{  H(\rho\mid\mu)+H\left(
\theta\mid\theta^{(j)}\otimes\hat{\mu}^{(j)}\right)  \right\}
\label{lim_balls}\\
& =H(\nu\mid\mu)+H\left(  \nu\mid\nu^{(j)}\otimes\hat{\mu}^{(j)}\right)
,\nonumber
\end{align}
where $\hat{\mu}^{(j)}\overset{\mathrm{def}}{=}\mu_{j+1}\otimes\cdots
\otimes\mu_{n}$. The inequality $\leq$ is evident by taking $\rho=\theta=\nu$,
and the opposite follows from the semicontinuity properties: One gets that for
a sequence $(\rho_{n},\theta_{n})$ with $\rho_{n}^{(j)}=\theta_{n}^{(j)}$ and
$\rho_{n},\theta_{n}\rightarrow\nu$, we have
\begin{align*}
\liminf_{n\rightarrow\infty}H\left(  \rho_{n}\mid\mu\right)   & \geq H(\nu
\mid\mu),\\
\liminf_{n\rightarrow\infty}H\left(  \theta_{n}\mid\theta_{n}^{(j)}\otimes
\hat{\mu}^{(j)}\right)   & \geq H\left(  \nu\mid\nu^{(j)}\otimes\hat{\mu
}^{(j)}\right)  ,
\end{align*}
the first inequality by the standard semi-continuity, and the second by Lemma
\ref{lower_semicontinuity}. This proves (\ref{lim_balls}).

Choose $\eta>0$ such that $H\left(  \nu^{(j)}\mid\mu^{(j)}\right)  <\Gamma
_{j}\log2-\eta$, for all $1\leq j\leq n$. By (\ref{lim_balls}) we may choose
$r$ small enough such that ${\operatorname{cl}}B_{r}(\nu)\subset V,$ and for
all $1\leq j\leq n$,
\begin{align*}
& \operatorname{inf}_{\rho,\theta\in{\operatorname{cl}}B_{r}(\nu):\rho
^{(j)}=\theta^{(j)}}\left\{  H(\rho\mid\mu)+H\left(  \theta\mid\theta
^{(j)}\otimes\hat{\mu}^{(j)}\right)  \right\} \\
& \geq H(\nu\mid\mu)+H\left(  \nu\mid\nu^{(j)}\otimes\hat{\mu}^{(j)}\right)
-\eta/2\\
& =2H(\nu\mid\mu)-H\left(  \nu^{(j)}\mid\mu^{(j)}\right)  -\eta/2\\
& \geq2H(\nu\mid\mu)-\Gamma_{j}\log2+{\eta/2}.
\end{align*}
For two indices $\alpha,\alpha^{\prime}$ we write $q(\alpha,\alpha^{\prime
})\overset{\mathrm{def}}{=}\max\left\{  j:\alpha^{(j)}=\alpha^{\prime
(j)}\right\}  $ with $\max\emptyset\overset{\mathrm{def}}{=}0$. Then
\begin{align*}
{{\mathbb{E}}}{M_{N}^{2}(U)}  & =\sum_{j=0}^{n}\sum_{\alpha,\alpha^{\prime
}:q(\alpha,\alpha^{\prime})=j}{\mathbb{P}}\left[  L_{N,\alpha}\in
U,L_{N,\alpha^{\prime}}\in U\right] \\
& =\sum_{\alpha,\alpha^{\prime}:q(\alpha,\alpha^{\prime})=0}{\mathbb{P}%
}\left[  L_{N,\alpha}\in U\right]  {\mathbb{P}}\left[  L_{N,\alpha^{\prime}%
}\in U\right] \\
& +\sum_{j=1}^{n}\sum_{\alpha,\alpha^{\prime}:q(\alpha,\alpha^{\prime}%
)=j}{\mathbb{P}}\left[  L_{N,\alpha}\in U,L_{N,\alpha^{\prime}}\in U\right] \\
& \leq{{\mathbb{E}}}[M_{N}({\operatorname{cl}}U)]^{2}+\\
& +\sum_{j=1}^{n}\sum_{\alpha,\alpha^{\prime}:q(\alpha,\alpha^{\prime}%
)=j}{\mathbb{P}}\left[  L_{N,\alpha}\in{\operatorname{cl}}U,L_{N,\alpha
^{\prime}}\in{\operatorname{cl}}U\right]  .
\end{align*}
We write the empirical measure in the form (\ref{splitting_empirical}), and
use Lemma \ref{ldp_empirical_measure_couple}. For any $1\leq j\leq n$ we have%
\begin{align*}
& \sum_{\alpha,\alpha^{\prime}:q(\alpha,\alpha^{\prime})=j}{\mathbb{P}}\left[
L_{N,\alpha}\in\operatorname{cl}U,L_{N,\alpha^{\prime}}\in\operatorname{cl}%
U\right] \\
& =2^{\Gamma_{j}N}2^{(1-\Gamma_{j})N}\left(  2^{(1-\Gamma_{j})N}-1\right)
{\mathbb{P}}\left[  L_{N,\alpha}\in\operatorname{cl}U,L_{N,\alpha^{\prime}}%
\in\operatorname{cl}U\right]  ,
\end{align*}
where on the right hand side $\alpha,\alpha^{\prime}$ is an arbitrary pair
with $q(\alpha,\alpha^{\prime})=j$. Using Lemma
\ref{ldp_empirical_measure_couple} we have
\begin{align*}
& {\mathbb{P}}\left[  L_{N,\alpha}\in\operatorname{cl}U,\;L_{N,\alpha}%
\in\operatorname{cl}U\right] \\
& \leq\operatorname{exp}\Bigg [-N\operatorname{inf}_{\rho,\theta
\in\operatorname{cl}U,\rho^{(j)}=\theta^{(j)}}\Big \{H\left(  \rho^{(j)}%
\mid\mu^{(j)}\right)  +\\
& +H\left(  \rho\mid\rho^{(j)}\otimes\hat{\mu}^{(j)}\right)  +H\left(
\theta\mid\theta^{(j)}\otimes\hat{\mu}^{(j)}\right)  \Big \}+{\frac{N\eta}{4}%
}\Bigg ]\\
& =\operatorname{exp}\left[  -N\operatorname{inf}_{\rho,\theta\in
\operatorname{cl}U,\rho^{(j)}=\theta^{(j)}}\left\{  H(\rho\mid\mu)+H\left(
\theta\mid\theta^{(j)}\otimes\hat{\mu}^{(j)}\right)  \right\}  +{\frac{N\eta
}{4}}\right] \\
& \leq2^{\Gamma_{j}N}\operatorname{exp}\left[  -2NH(\nu\mid\mu)-{\frac{N\eta
}{4}}\right]  ,
\end{align*}
and thus
\[
\sum_{\alpha,\alpha^{\prime}:q(\alpha,\alpha^{\prime})=j}{\mathbb{P}}\left[
L_{N,\alpha}\in\operatorname{cl}U,\;L_{N,\alpha}\in\operatorname{cl}U\right]
\leq2^{2N}\operatorname{exp}\left[  -2NH(\nu\mid\mu)-{\frac{N\eta}{4}}\right]
.
\]
Combining, we obtain by taking $\chi=\eta/16$ in (\ref{Est_Exp_below})%
\[
\operatorname{var}\left[  M_{N}(U)\right]  \leq2^{2N}\operatorname{exp}\left[
-2NH(\nu\mid\mu)-{\frac{N\eta}{4}}\right]  \leq\mathrm{e}^{-N\eta
/8}{{\mathbb{E}}}[M_{N}(U)]^{2},
\]
which proves our claim.
\end{proof}

\begin{proof}
[Proof of Theorem \ref{Th_GREM_perceptron}]We set
\[
{\mathcal{G}}\overset{\mathrm{def}}{=}\left\{  \nu\in{\mathcal{M}}_{1}%
^{+}(S^{n}):H\left(  \nu^{(j)}\mid\mu^{(j)}\right)  \leq\Gamma_{j}%
\log2,\ j=1,\dots,n\right\}  ,
\]
which is a compact set.

\textit{Step 1.} We first prove the lower bound. By compactness of
${\mathcal{G}}$ and the semicontinuity of $H$ there exists $\nu_{0}%
\in{\mathcal{G}}$ such that
\[
\sup_{\nu\in{\mathcal{G}}}\left\{  \Phi(\nu)-H(\nu\mid\mu)\right\}  =\Phi
(\nu_{0})-H(\nu_{0}\mid\mu).
\]
We set $\nu_{\lambda}\overset{\mathrm{def}}{=}(1-\lambda)\nu_{0}+\lambda\mu$
for $0<\lambda<1$. By convexity of $H(\nu\mid\mu)$ in $\nu$ we see that
$H\left(  \nu_{\lambda}^{(j)}\mid\mu^{(j)}\right)  <\Gamma_{j}\log2$ for all
$1\leq j\leq n$. Furthermore $\nu_{\lambda}\rightarrow\nu_{0}$ weakly as
$\lambda\rightarrow0$, and $\Phi(\nu_{\lambda})\rightarrow\Phi(\nu_{0})$,
$H(\nu_{\lambda}\mid\mu)\rightarrow H(\nu_{0}\mid\mu)$.

Given $\varepsilon>0$ we choose $\lambda>0$ such that
\[
\Phi(\nu_{\lambda})-H(\nu_{\lambda}\mid\mu)\geq\Phi(\nu_{0})-H(\nu_{0}\mid
\mu)-\varepsilon.
\]
By the continuity of $\Phi$ and Lemma \ref{second_mom} we find a neighborhood
$U$ of $\nu_{\lambda}$, and $\delta>0$ such that
\[
\Phi(\theta)-\Phi(\nu_{\lambda})\leq\varepsilon,\ \theta\in U,
\]
and
\[
{\mathbb{P}}\left[  M_{N}(U)\leq2^{N}\operatorname{exp}\left[  -NH(\nu
_{\lambda}\mid\mu)-N\varepsilon\right]  \right]  \leq\mathrm{e}^{-\delta N},
\]
Then, with probability greater than $1-\mathrm{e}^{-\delta N}$,%
\begin{align*}
Z_{N}  & =2^{-N}\sum_{\alpha}\operatorname{exp}\left[  N\Phi(L_{N,\alpha
})\right] \\
& \geq2^{-N}\sum_{\alpha:L_{N,\alpha}\in U}\operatorname{exp}\left[
N\Phi(L_{N,\alpha})\right] \\
& \geq\operatorname{exp}\left[  N\Phi(\nu_{\lambda})-N\varepsilon\right]
\operatorname{exp}\left[  -NH(\nu_{\lambda}\mid\mu)-N\varepsilon\right] \\
& \geq\operatorname{exp}\left[  N\sup_{\nu\in{\mathcal{G}}}\left\{  \Phi
(\nu)-H(\nu\mid\mu)\right\}  -3N\varepsilon\right]  .
\end{align*}
By Borel-Cantelli, we therefore get, as $\varepsilon$ is arbitrary,
\[
\liminf_{N\rightarrow\infty}{\frac{1}{N}}\log Z_{N}\geq\sup_{\nu
\in{\mathcal{G}}}\left\{  \Phi(\nu)-H(\nu\mid\mu)\right\}
\]
almost surely.

\textit{Step 2.} We prove the upper bound. Let again $\varepsilon>0$ and set
\[
\overline{{\mathcal{G}}}\overset{\mathrm{def}}{=}\{\nu:H(\nu\mid\mu)\leq
\log2\}.
\]
If $\nu\in{\mathcal{G}}$ we choose $r_{\nu}>0$ such that $\left\vert
\Phi(\theta)-\Phi(\nu)\right\vert \leq\varepsilon$, $\theta\in B_{r_{\nu}}%
(\nu)$ and
\[
{\mathbb{P}}\left[  M_{N}(B_{r_{\nu}}(\nu))\geq2^{N}\operatorname{exp}\left[
-NH(\nu\mid\mu)+N\varepsilon\right]  \right]  \leq\mathrm{e}^{-N\delta_{\nu}},
\]
for some $\delta_{\nu}>0$ and large enough $N$ (using Lemma \ref{very_many}).
If $\nu\in\overline{{\mathcal{G}}}\setminus{\mathcal{G}}$ we choose $r_{\nu}$
such that $\left\vert \Phi(\theta)-\Phi(\nu)\right\vert \leq\varepsilon$,
$\theta\in B_{r_{\nu}}(\nu)$, and
\begin{equation}
{\mathbb{P}}\left[  M_{N}(B_{r_{\nu}}(\nu))\neq0\right]  \leq\mathrm{e}%
^{-N\delta_{\nu}},\label{still_existent_two}%
\end{equation}
again for large enough $N$ (and by Lemma \ref{still_existent}). As
$\overline{{\mathcal{G}}}$ is compact, we can cover it by a finite union of
such balls, i.e.
\[
\overline{{\mathcal{G}}}\subset U\overset{\mathrm{def}}{=}\bigcup_{j=1}%
^{m}B_{r_{j}}(\nu_{j}),
\]
where $r_{j}\overset{\mathrm{def}}{=}r_{\nu_{j}}$. We also set $\delta
\overset{\mathrm{def}}{=}\min_{j}\delta_{\nu_{j}}$. We then estimate
\begin{equation}
Z_{N}\leq2^{-N}\sum_{l=1}^{m}\sum_{\alpha:L_{N,\alpha}\in B_{r_{l}}(\nu_{l}%
)}\operatorname{exp}\left[  N\Phi(L_{N,\alpha})\right]  +2^{-N}\sum
_{\alpha:L_{N,\alpha}\notin U}\operatorname{exp}\left[  N\Phi(L_{N,\alpha
})\right]  .\label{upper_bound_abstract}%
\end{equation}
we first claim that almost surely the second summand vanishes provided $N$ is
large enough, i.e. that there is no $\alpha$ with $L_{N,\alpha}\notin U$. By
Sanov's theorem, we have
\[
\limsup_{N\rightarrow\infty}{\frac{1}{N}}\log{\mathbb{P}}\left[  L_{N,\alpha
}\notin U\right]  \leq-\operatorname{inf}_{\nu\notin U}H(\nu\mid\mu)<-\log2.
\]
Therefore, almost surely, there is no $\alpha$ with $L_{N,\alpha}\notin U$,
and therefore the second summand in (\ref{upper_bound_abstract}) vanishes for
large enough $N$, almost surely. The same applies to those summands in the
first part for which $\nu_{l}\notin{\mathcal{G}}$, using
(\ref{still_existent_two}). We therefore have, almost surely, for large enough
$N$,
\begin{align*}
Z_{N}  & \leq2^{-N}\sum_{l:\nu_{l}\in{\mathcal{G}}}\sum_{\alpha:L_{N,\alpha
}\in B_{r_{l}}(\nu_{l})}\operatorname{exp}\left[  N\Phi(L_{N,\alpha})\right]
\\
& \leq\mathrm{e}^{N\varepsilon}\sum_{l:\nu_{l}\in{\mathcal{G}}}%
\operatorname{exp}\left[  N\Phi(\nu_{l})\right]  M_{N}(B_{r_{l}}(\nu_{l}))\\
& \leq\mathrm{e}^{2N\varepsilon}\sum_{l:\nu_{l}\in{\mathcal{G}}}%
\operatorname{exp}\left[  N\Phi(\nu_{l})\right]  \operatorname{exp}\left[
-NH(\nu_{l}\mid\mu)\right] \\
& \leq\mathrm{e}^{2N\varepsilon}m\operatorname{exp}\left[  N\sup_{\nu
\in{\mathcal{G}}}\left\{  \Phi(\nu)-H(\nu\mid\mu)\right\}  \right]  .
\end{align*}
As $\varepsilon$ is arbitrary, we get
\[
\limsup_{N\rightarrow\infty}{\frac{1}{N}}\log Z_{N}\leq\sup_{\nu
\in{\mathcal{G}}}\left[  \Phi(\nu)-H(\nu\mid\mu)\right]  .
\]
This finishes the proof of Theorem \ref{Th_GREM_perceptron}.
\end{proof}

\subsection{The dual representation. Proof of the Theorem
\ref{Th_Parisi_formula}}

We define a family $\mathcal{G}\left(  \phi\right)  =\left\{  G_{\phi
,\mathbf{m}}\right\}  $ of probability distributions on $S^{n}$ which depend
on the parameter $\mathbf{m}=\left(  m_{1},\ldots,m_{n}\right)  \in\Delta.$
The probability measure $G=G_{\phi,\mathbf{m}}$ is described by a
\textquotedblleft starting\textquotedblright\ measure $\gamma$ on $S,$ and for
$2\leq j\leq n$ Markov kernels $K_{j}$ from $S^{j-1}$ to $S,$ so that $G$ is
the semi-direct product%
\[
G=\gamma\otimes K_{2}\otimes\cdots\otimes K_{n}.
\]%
\[
\gamma\left(  dx\right)  \overset{\mathrm{def}}{=}\frac{\exp\left[  m_{1}%
\phi_{1}\left(  x\right)  \right]  \mu_{1}\left(  dx\right)  }{\exp\left[
m_{1}\phi_{0}\right]  },
\]%
\[
K_{j}\left(  \mathbf{x}^{\left(  j-1\right)  },dx_{j}\right)  \overset
{\mathrm{def}}{=}\frac{\exp\left[  m_{j}\phi_{j}\left(  \mathbf{x}^{\left(
j\right)  }\right)  \right]  \mu_{j}\left(  dx_{j}\right)  }{\exp\left[
m_{j}\phi_{j-1}\left(  \mathbf{x}^{\left(  j-1\right)  }\right)  \right]  },
\]
where we write $\mathbf{x}^{\left(  j\right)  }\overset{\mathrm{def}}%
{=}\left(  x_{j},\ldots,x_{j}\right)  .$ Remember the definition of the
function $\phi_{j}:S^{j}\rightarrow\mathbb{R}$ in (\ref{Def_phin}),
(\ref{Def_phij}). It should be remarked that these objects are defined for all
$\mathbf{m}\in\mathbb{R}^{n}$, and not just for $\mathbf{m}\in\Delta.$ We also
write%
\[
G^{\left(  j\right)  }\overset{\mathrm{def}}{=}\gamma\otimes K_{2}%
\otimes\cdots\otimes K_{j}%
\]
which is the marginal of $G$ on $S^{j}.$ In order to emphasize the dependence
on $\mathbf{m},$ we occasionally will write $\phi_{j,\mathbf{m}}%
,\ \gamma_{\mathbf{m}},\ K_{j,\mathbf{m}}$ etc.

We remark that by a simple computation%
\begin{gather}
\int H\left(  K_{j}\left(  \mathbf{x}^{\left(  j-1\right)  },\cdot\right)
\mid\mu_{j}\right)  G^{\left(  j-1\right)  }\left(  d\mathbf{x}^{\left(
j-1\right)  }\right) \label{PAR1}\\
=m_{j}\left[  \int\phi_{j}dG^{\left(  j\right)  }-\int\phi_{j-1}dG^{\left(
j-1\right)  }\right]  .\nonumber
\end{gather}
$\phi_{j},\ldots,\phi_{n}$ do not depend on $m_{j},$ but $\phi_{0},\ldots
,\phi_{j-1}$ do. Differentiating the equation%
\[
\mathrm{e}^{m_{r+1}\phi_{r}}=\int\mathrm{e}^{m_{r+1}\phi_{r+1}}d\mu_{r+1}%
\]
with respect to $m_{j},$ we get for $0\leq r\leq j-2$%
\begin{equation}
\frac{\partial\phi_{r}\left(  \mathbf{x}^{\left(  r\right)  }\right)
}{\partial m_{j}}=\int\frac{\partial\phi_{r+1}\left(  \mathbf{x}^{\left(
r\right)  },x_{r+1}\right)  }{\partial m_{j}}K_{r+1}\left(  d\mathbf{x}%
^{\left(  r\right)  },x_{r+1}\right)  ,\label{PAR2}%
\end{equation}
and for $r=j-1$%
\[
\phi_{j-1}\mathrm{e}^{m_{j}\phi_{j}}+m_{j}\frac{\partial\phi_{j-1}}{\partial
m_{j}}\mathrm{e}^{m_{j}\phi_{j}}=\int\phi_{j}\mathrm{e}^{m_{j}\phi_{j}}%
d\mu_{j},
\]
i.e.%
\[
\frac{\partial\phi_{j-1}}{\partial m_{j}}\left(  \mathbf{x}^{\left(  j\right)
}\right)  =\frac{1}{m_{j}}\left[  \int\phi_{j}\left(  \mathbf{x}^{\left(
j-1\right)  },x_{j}\right)  K_{j}\left(  \mathbf{x}^{\left(  j-1\right)
},dx_{j}\right)  -\phi_{j-1}\left(  \mathbf{x}^{\left(  j-1\right)  }\right)
\right]  .
\]
Combining that with (\ref{PAR1}), (\ref{PAR2}) we get%
\begin{align}
\frac{\partial\phi_{0,\mathbf{m}}}{\partial m_{j}}  & =\frac{1}{m_{j}}\left[
\int\phi_{j}dG^{\left(  j\right)  }-\int\phi_{j-1}dG^{\left(  j-1\right)
}\right] \label{PAR3}\\
& =\frac{1}{m_{j}^{2}}\int H\left(  K_{j}\left(  \mathbf{x}^{\left(
j-1\right)  },\cdot\right)  \mid\mu_{j}\right)  G^{\left(  j-1\right)
}\left(  d\mathbf{x}^{\left(  j-1\right)  }\right)  .\nonumber
\end{align}

Theorem \ref{Th_Parisi_formula} is immediate from the following result:

\begin{proposition}
\label{Th_Gibbs_maximizer}Assume that $\phi:S^{n}\rightarrow{\mathbb{R}}$ is
bounded and continuous. Then there is a unique measure $\nu$ maximizing
$\operatorname{Gibbs}\left(  \nu,\phi\right)  $ under the constraint $\nu
\in\bigcap\nolimits_{j=1}^{n}\mathcal{R}_{j}.$ This measure is of the form
$\nu=G_{\phi,\mathbf{m}}$ where $\mathbf{m}$ is the unique element in $\Delta$
minimizing (\ref{Variationformula3}). For this $\mathbf{m},$ we have%
\begin{equation}
\operatorname{Gibbs}\left(  G,\phi\right)  =\operatorname{Parisi}\left(
\phi,\mathbf{m}\right)  .\label{Gibbs=Parisi}%
\end{equation}

\end{proposition}

\begin{proof}
From strict convexity of the relative entropy, and the fact that
$\bigcap\nolimits_{j=1}^{n}\mathcal{R}_{j}$ is compact and convex, it follows
that there is a unique maximizer $\nu$ of $\operatorname{Gibbs}\left(
\nu,\phi\right)  $ under this constraint.

Also, a straightforward application of H\"{o}lder's inequality shows that
$\operatorname{Parisi}\left(  \phi,\mathbf{m}\right)  $ is a strictly convex
function in the variables $1/m_{j}.$ Therefore, it follows that there is a
uniquely attained minimum of $\operatorname{Parisi}\left(  \phi,\mathbf{m}%
\right)  $ as a function of $\mathbf{m}\in\Delta.$ This minimizing
$\mathbf{m}=\left(  m_{1},\ldots,m_{n}\right)  $, we can be split into
subblocks of equal values: There is a number $K,\ 0\leq K\leq n,$ and indices
$0<j_{1}<j_{2}<\cdots<j_{K}\leq n$ such that%
\begin{align*}
0  & <m_{1}=\cdots=m_{j_{1}}<m_{j_{1}+1}=\cdots=m_{j_{2}}\\
& <m_{j_{2}+1}\cdots<m_{j_{K-1}+1}=\cdots=m_{j_{K}}\\
& <m_{j_{K}+1}=\cdots m_{n}=1.
\end{align*}
$K=0$ just means that all $m_{i}=1.$ If $j_{K}=n,$ then all $m_{i}$ are $<1.$
We write $G=G_{\phi,\mathbf{m}}.$

>From (\ref{PAR3}), we immediately have%
\begin{equation}
\frac{\partial\operatorname{Parisi}\left(  \phi,\mathbf{m}\right)  }{\partial
m_{j}}=\frac{1}{m_{j}^{2}}\left[  \int H\left(  K_{j}\left(  \mathbf{x}%
^{\left(  j-1\right)  },\cdot\right)  \mid\mu_{j}\right)  G^{\left(
j-1\right)  }\left(  d\mathbf{x}^{\left(  j-1\right)  }\right)  -\gamma
_{j}\log2\right]  .\label{PAR_Derivative}%
\end{equation}

Set $d_{j}\overset{\mathrm{def}}{=}\int H\left(  K_{j}\left(  \mathbf{x}%
^{\left(  j-1\right)  },\cdot\right)  \mid\mu_{j}\right)  G_{\mathbf{m}%
}^{\left(  j-1\right)  }\left(  d\mathbf{x}^{\left(  j-1\right)  }\right)
.$We use (\ref{PAR_Derivative}) and the minimality of $\operatorname{Parisi}%
\left(  \phi,\mathbf{\cdot}\right)  $ at $\mathbf{m.}$ We can perturb
$\mathbf{m}$ by moving a whole block $m_{j_{r}+1}=\cdots=m_{j_{r+1}}$ up and
down locally, without leaving $\Delta,$ provided it is not the possibly
present block of values $1.$ This leads to%
\[
\sum_{i=j_{r}+1}^{j_{r+1}}d_{i}=\log2\sum_{i=j_{r}+1}^{j_{r+1}}\gamma_{i}.
\]
Furthermore, we can always move first parts of blocks, say $m_{j_{r}+1}%
=\cdots=m_{k},\ k\leq j_{r+1}$ locally down, without leaving $\Delta,$ so that
we get%
\[
\sum_{i=j_{r}+1}^{j_{k}}d_{i}\leq\log2\sum_{i=j_{r}+1}^{j_{k}}\gamma_{i}.
\]
These two observations imply%
\begin{equation}
G\in\bigcap_{j=1}^{n}\mathcal{R}_{j}\cap\bigcap_{r=1}^{K}\mathcal{R}_{j_{r}%
}^{=}.\label{PAR4}%
\end{equation}

We next prove%
\begin{equation}
\operatorname{Gibbs}\left(  \nu,\phi\right)  \leq\operatorname{Gibbs}\left(
G,\phi\right) \label{PAR5}%
\end{equation}
for any $\nu\in\bigcap_{j=1}^{n}\mathcal{R}_{j}.$

We first prove the case $n=1.$ If $m<1,$ then%
\[
H\left(  G\mid\mu\right)  =\log2\geq H\left(  \nu\mid\mu\right)
\]
by (\ref{PAR4}) and the assumption $\nu\in\mathcal{R}_{1}.$ Therefore, in any
case%
\begin{align*}
\operatorname{Gibbs}\left(  G,\phi\right)  -\operatorname{Gibbs}\left(
\nu,\phi\right)   & \geq\int\phi dG-\frac{1}{m}H\left(  G\mid\mu\right) \\
& -\left[  \int\phi d\nu-\frac{1}{m}H\left(  \nu\mid\mu\right)  \right] \\
& =\frac{1}{m}H\left(  \nu\mid G\right)  \geq0
\end{align*}

The general case follows by a slight extension of the above argument. Put%
\[
D_{k}\overset{\mathrm{def}}{=}\int\phi_{k}dG^{\left(  k\right)  }-\frac
{1}{m_{k+1}}H\left(  G^{\left(  k\right)  }\mid\mu^{\left(  k\right)
}\right)  -\int\phi_{k}d\nu^{\left(  k\right)  }+\frac{1}{m_{k+1}}H\left(
\nu^{\left(  k\right)  }\mid\mu^{\left(  k\right)  }\right)  ,
\]
$D_{0}\overset{\mathrm{def}}{=}0,D_{n}=\operatorname{Gibbs}\left(
G,\phi\right)  -\operatorname{Gibbs}\left(  \nu,\phi\right)  .$ We prove
$D_{k-1}\leq D_{k}$ for all $k,$ so that the claim follows. Remark that as
above in the $n=1$ case, if $m_{k}<m_{k+1},$ then $H\left(  G^{\left(
k+1\right)  }\mid\mu^{\left(  k+1\right)  }\right)  =\Gamma_{k}\log2,$ and
therefore, in any case%
\begin{align*}
D_{k}  & \geq\int\phi_{k}dG^{\left(  k\right)  }-\frac{1}{m_{k}}H\left(
G^{\left(  k\right)  }\mid\mu^{\left(  k\right)  }\right)  -\int\phi_{k}%
d\nu^{\left(  k\right)  }+\frac{1}{m_{k}}H\left(  \nu^{\left(  k\right)  }%
\mid\mu^{\left(  k\right)  }\right) \\
& =\int\phi_{k-1}dG^{\left(  k-1\right)  }-\frac{1}{m_{k}}H\left(  G^{\left(
k-1\right)  }\mid\mu^{\left(  k-1\right)  }\right)  -\int\phi_{k}d\nu^{\left(
k\right)  }+\frac{1}{m_{k}}H\left(  \nu^{\left(  k\right)  }\mid\mu^{\left(
k\right)  }\right)  .
\end{align*}
As%
\begin{align*}
& H\left(  \nu^{\left(  k\right)  }\mid\mu^{\left(  k\right)  }\right)
-m_{k}\int\phi_{k}d\nu^{\left(  k\right)  }+m_{k}\int\phi_{k-1}d\nu^{\left(
k-1\right)  }\\
& =H\left(  \nu^{\left(  k-1\right)  }\mid\mu^{\left(  k-1\right)  }\right)
+\int\log\frac{\nu^{\left(  k\right)  }\left(  dx_{k}\mid\mathbf{x}^{\left(
k-1\right)  }\right)  \mathrm{e}^{m_{k}\phi_{k-1}\left(  \mathbf{x}^{\left(
k-1\right)  }\right)  }}{\mu_{k}\left(  dx_{k}\right)  \mathrm{e}^{m_{k}%
\phi_{k}\left(  \mathbf{x}^{\left(  k\right)  }\right)  }}\nu^{\left(
k\right)  }\left(  d\mathbf{x}^{\left(  k\right)  }\right) \\
& \geq H\left(  \nu^{\left(  k-1\right)  }\mid\mu^{\left(  k-1\right)
}\right)  ,
\end{align*}
(\ref{PAR5}) is proved.

(\ref{PAR4}) and (\ref{PAR5}) identify $G=G_{\phi,\mathbf{m}}$ as the unique
maximizer of $G\left(  \cdot,\phi\right)  $ under the constraint
$\bigcap\nolimits_{j=1}^{n}\mathcal{R}_{j}.$

The identification (\ref{Gibbs=Parisi}) comes by a straightforward computation.
\end{proof}

\end{document}